\documentclass[12pt]{amsart}

\usepackage{amssymb}
\usepackage{cancel}
\usepackage{graphicx}
\usepackage{enumerate}

\usepackage{epsf}

\newtheorem{theorem}{Theorem}

\newtheorem{corollary}{Corollary}

\newcommand{\ff}[2]{ \ensuremath{ {#1}^{\underline{#2}} }}

\newcommand{\floor}[1]{ \ensuremath{ \left\lfloor #1\right\rfloor}}

\newcommand{\stone}[2]{\ensuremath{\left[\begin{matrix}{#1}\\{#2}\end{matrix}\right]}}
\newcommand{\sttwo}[2]{\ensuremath{\left\{\begin{matrix}{#1}\\{#2}\end{matrix}\right\}}}

\title{Falling Factorials, Generating Functions, and Conjoint Ranking Tables}

\author{Brad Osgood \and William Wu}  
\address{Information Systems Lab\\Stanford University}

\begin{document}

\bibliographystyle{amsplain}

\maketitle   

\begin{abstract}
We investigate the coefficients generated by expressing the falling factorial $\ff{(xy)}{k}$ as a linear combination of falling factorial products $\ff{x}{l}\ff{y}{m}$ for $l,m =1, \dots, k$. Algebraic and combinatoric properties are discussed, some in relation to Stirling numbers. 
\end{abstract}

\section{Introduction}

Let $\ff{x}{k}$ denote the \emph{falling factorial power},
\[
\ff{x}{k}=x(x-1)\ldots(x-k+1)\,,\quad  k\in \mathbb{N},  k \ge 1\,.
\] 
We think of $x \in \mathbb{R}$, but the definition can make sense in more general settings. Problems from discrete Fourier analysis --  distant from the topics considered here -- led us to falling factorial powers of products expressed as
\begin{equation}
\label{eq:main}
\ff{(xy)}{k} = \sum_{l,m=1}^k c^{(k)}_{l,m} \ff{x}{l} \ff{y}{m}\,.
\end{equation}
There is the obvious symmetry $c^{(k)}_{l,m} = c^{(k)}_{m,l}$.  Since $c^{(1)}_{1,1}=1$ the interest begins when $k \ge 2$. 
For example,
\[
\begin{aligned}
\ff{(xy)}{2} &=\ff{x}{1}\ff{y}{2}+\ff{x}{2}\ff{y}{1}+\ff{x}{2}\ff{y}{2}\,,\\
\ff{(xy)}{3}&= \ff{x}{1}\ff{y}{3}+\ff{x}{3}\ff{y}{1}+6\ff{x}{2}\ff{y}{2}+3\ff{x}{2}\ff{y}{3}+3\ff{x}{3}\ff{y}{2}+\ff{x}{3}\ff{y}{3}\,.
\end{aligned} 
\]
Note that 
\[
\begin{aligned}
&c^{(2)}_{1,1}=0\\
&c^{(3)}_{1,1}=0\,,\quad c^{(3)}_{1,2}=c^{(3)}_{2,1}=0\,.
\end{aligned}
\]
and that the coefficients that do appear are positive.

Here are the values for $c^{(9)}_{l,m}$ displayed as a symmetric matrix:
\begin{equation} \label{eq:C9}
\begin{pmatrix}
0 &	0&	0&	0&	0&	0&	0&	0&	1\\
0&	0&	0&	0&	15120&	40320&	24192&	4608&	255\\
0&	0&	10080&	544320&	1958040&	1796760&	588168&	74124&	3025\\
0&	0&	544320&	6108480&	12267360&	7988904&	2066232&	218484&	7770\\
0&	15120&	1958040&	12267360&	18329850&	9874746&	2229402&	212436&	6951\\
0&	40320&	1796760&	7988904&	9874746&	4690350&	965790&	85680&	2646\\
0&	24192&	588168&	2066232&	2229402&	965790&	185766&	15624&	462\\
0&	4608&	74124&	218484&	212436&	85680&	15624&	1260	& 36\\
1&	255&	3025&	7770&	6951&	2646&	462&	36&	1
\end{pmatrix}
\end{equation}


The numbers $c^{(k)}_{l,m}$ 
have a number of interesting properties that are the subject of the present paper.  We found a recurrence relation, several closed-form expressions (which appear rather different from each other), a natural combinatorial interpretation in terms of conjoint ranking tables, and we can extend these results to products of more than two variables. 

All of this was new to us, but we probably should not have been surprised to discover, later, that the combinatorics had essentially been considered in other contexts. To make the approach taken here self-contained and more readable we have rederived (briefly) some of these earlier results, with references. We also mention several questions that we were unable to answer.

We would like to thank Jiehua Chen, John Gill, Michael Godfrey, and Donald Knuth for their comments.

\section{Stirling Numbers and Falling Factorial Powers} \label{section:stirling}

One can solve for the coefficients $c^{(k)}_{l,m}$ in any particular case, but that there should generally be such an expansion emerges from the connection between falling factorial powers, ordinary powers  and Stirling numbers;  see \cite{knuth:concrete}, whose notation we follow. 

In combinatorics one defines the (unsigned) Stirling numbers of the first kind by
\[
\stone{k}{l}
= \text{the number of permutations of $k$ letters which consist of $l$ disjoint cycles.} 
\]
In particular
\begin{equation} \label{eq:stirling-1-special}
\stone{k}{k} = 1\,,\quad \stone{k}{l} = 0 \quad \text{if $k<l$.}
\end{equation} 
For us,  the important fact is
\begin{equation}
\label{eq:stir1-powers}
\ff{x}{k} = \sum_{l=1}^n(-1)^{k-l}\stone{k}{l} x^l\,.
\end{equation}
Stirling numbers of the second kind, denoted by $\sttwo{k}{l}$, are defined by
\[
\sttwo{k}{l} = \text{the number of partitions of a set of $k$ elements into $l$ nonempty subsets}.
\]
Here
\begin{equation} \label{eq:stirling-2-special}
\sttwo{k}{1}=\sttwo{k}{k} = 1\,,\quad \sttwo{k}{l} = 0 \quad \text{if $k<l$.}
\end{equation}
The special values in \eqref{eq:stirling-1-special} and \eqref{eq:stirling-2-special} will come up in  adjusting summation indices. 

Corresponding to \eqref{eq:stir1-powers} one has
\begin{equation}
\label{eq:stir2-powers}
x^k = \sum_{l=1}^n \sttwo{k}{l} \ff{x}{l}.
\end{equation}
We included the phrase `Generating Functions' in the title of the paper because the sequences $\stone{k}{l}$, $\sttwo{k}{l}$, and $c^{(k)}_{l,m}$ each count something and each appears in an expansion in powers, \eqref{eq:stir1-powers}, \eqref{eq:stir2-powers}, and \eqref{eq:main}, much like classical generating functions. 

\bigskip

Proceeding with the derivation of \eqref{eq:main},  from \eqref{eq:stir1-powers} and \eqref{eq:stir2-powers} we have
\begin{eqnarray*}
\ff{(xy)}{k} 
&=&  \sum_{p=1}^k (-1)^{k-p}\stone{k}{p} (xy)^p  =  \sum_{p=1}^k(-1)^{k-p} \stone{k}{p}x^p y^p \\
&=&  \sum_{p=1}^k(-1)^{k-p} \stone{k}{p} \left( \sum_{l=1}^p \sttwo{p}{l} \ff{x}{l} \right)\left(  \sum_{m=1}^p \sttwo{p}{m} \ff{y}{m} \right) \\
&=& \sum_{p=1}^k \sum_{l=1}^p \sum_{m=1}^p  (-1)^{k-p}\stone{k}{p} \sttwo{p}{l}\sttwo{p}{m}\ff{x}{l} \ff{y}{m}  
\end{eqnarray*}
It is now a matter of swapping the summations and using a property of Stirling numbers. First, employing Knuth's version of Iverson's convention for sums we write
\[
\begin{aligned}
 \sum_{p=1}^k \sum_{l=1}^p \sum_{m=1}^p  (-1)^{k-p}&\stone{k}{p} \sttwo{p}{l}\sttwo{p}{m}\ff{x}{l} \ff{y}{m}\\
 & = \sum_{l,m,p}(-1)^{k-p}\stone{k}{p} \sttwo{p}{l}\sttwo{p}{m}\ff{x}{l} \ff{y}{m}\,[1 \le l\le p][1 \le m \le p][1 \le p \le k]\,.
 \end{aligned}
 \]
But now
\[
[1 \le p \le k] [1\le l \le p][1 \le m \le p]=[1 \le l,m\le k][\max(l,m)\le p \le k]\,,
\]
and therefore
\[
 \sum_{p=1}^k \sum_{l=1}^p \sum_{m=1}^p  (-1)^{k-p}\stone{k}{p} \sttwo{p}{l}\sttwo{p}{m}\ff{x}{l} \ff{y}{m} = \sum_{l,m=1}^k \sum_{p=\max(l,m)}^k (-1)^{k-p}\stone{k}{p} \sttwo{p}{l}\sttwo{p}{m}\ff{x}{l} \ff{y}{m} \,.
 \]
Next, since $\sttwo{p}{l}=0$ for $p<l$ and $\sttwo{p}{m}=0$ for $p<m$ we can allow the innermost summation to begin at $p=1$ without affecting the result. Thus
\[
\ff{(xy)}{k} =  \sum_{l=1}^k\sum_{m=1}^k \sum_{p=1}^k (-1)^{k-p}\stone{k}{p} \sttwo{p}{l}\sttwo{p}{m}\ff{x}{l} \ff{y}{m} = \sum_{l,m=1}^k c_{l,m}^{(k)} \ff{x}{l} \ff{y}{m} \,,
\]
where
\begin{equation} \label{eq:c-formula-1}
c_{l,m}^{(k)} = \sum_{p=1}^k (-1)^{k-p}\stone{k}{p}  \sttwo{p}{l} \sttwo{p}{m} \,.
\end{equation}
This establishes the expansion \eqref{eq:main} and provides a formula for the coefficients. It is not clear from this expression that the $c$'s are nonnegative. We will deduce this from a combinatorial characterization in Section \ref{section:combinatorics}.


When $l = k$,
\begin{eqnarray*}
c_{k,m}^{(k)} 
&=& \sum_{p=1}^k (-1)^{k-p}\stone{k}{p} \sttwo{p}{k} \sttwo{p}{m}  \\
&=& \stone{k}{k} \sttwo{k}{k} \sttwo{k}{m} \\
&=& \sttwo{k}{m}\,.
\end{eqnarray*}	
Thus the Stirling numbers of the second kind appear in the last row (or column) of the matrix of the $c$'s, as we see in \eqref{eq:C9}.

There are a few more arithmetic properties of the $c^{(k)}_{l,m}$ that we wish to note, expressed most easily in terms of the symmetric matrix $C^{(k)}$ whose ($l,m$)-entry is $c^{(k)}_{l,m}$. Let $\Sigma^{(k)}$ be the $k \times k$ diagonal matrix whose entries are
\[
\Sigma^{(k)}_{lm} = (-1)^{k-l}\stone{k}{l}\delta_{lm}\,.
\]
Let $S_2^{(k)}$ be the upper triangular matrix with nonzero entries
\[
(S_2^{(k)})_{lm} = \sttwo{m}{l}\,,\quad l \le m\,.
\]
One can check that 
\[
C^{(k)} = S_2^{(k)}\Sigma^{(k)}(S_2^{(k)})^T\,.
\]
In turn it follows that $C^{(k)}$ is invertible and has the same signature as $\Sigma^{(k)}$. Since the diagonal entries in $\Sigma^{(k)}$ alternate sign, $C^{(k)}$ has $\lceil k/2\rceil$ positive eigenvalues and $k - \lceil k/2\rceil$ negative eigenvalues. Moreover, the diagonal entries of $S_2^{(k)}$ are $1$'s so $\det S_2^{(k)} = 1$, and then
\[
\det C^{(k)} = \prod_{l=1}^k(-1)^{k-l}\stone{k}{l}\,.
\]

For all $k$ it appears that the rows (columns), the diagonals, and the anti-diagonals of $C^{(k)}$ are all unimodal; this is quite visible for $k=9$ in \eqref{eq:C9}. Stronger than that, numerical evidence suggests all are log-concave. We are only able to show this for the anti-diagonals,
\[
\dots, c^{(k)}_{l+1,m-1}, c^{(k)}_{l,m},c^{(k)}_{l-1,m+1}, \dots\,.
\]
Recall that a sequence $\{a_i\}$ is log-concave if
\[
a_ia_{i-2} \le a_{i-1}^2\,.
\]
This is easily seen to be equivalent to the condition
\[
a_la_m \le a_{l-1}a_{m+1}\quad \text{for all $m \le l-2$.}
\]
It is known, \cite{stanley1989lca}, that the Stirling numbers $\sttwo{k}{l}$ are log-concave in $l$ for fixed $k$. Thus
\[
\sttwo{p}{l-1}\sttwo{p}{m+1} \le \sttwo{p}{l}\sttwo{p}{m} \quad \text{and}\quad \sttwo{q}{l+1}\sttwo{q}{m-1} \le \sttwo{q}{l}\sttwo{q}{m}\,,
\]
and using this in \eqref{eq:c-formula-1} we have
\[
\begin{aligned}
c^{(k)}_{l-1,m+1}c^{(k)}_{l+1,m-1} & = \sum_{p,q=1}^k (-1)^{k-p}(-1)^{k-q}\stone{k}{p}\stone{k}{q}
\sttwo{p}{l-1}\sttwo{p}{m+1}\sttwo{q}{l+1}\sttwo{q}{m-1}\\
&\le \sum_{p,q=1}^k (-1)^{k-p}(-1)^{k-q}\stone{k}{p}\stone{k}{q}
\sttwo{p}{l}\sttwo{p}{m}\sttwo{q}{l}\sttwo{q}{m}\\
&=(c^{(k)}_{l,m})^2\,,
\end{aligned}
\]
which is what we are required to show.

\section{Recurrence} \label{section:recurrence}

Falling factorial powers satisfy
\[
\Delta \ff{x}{n} = (n-1) \ff{x}{{n-1}}\,,
\]
where $\Delta$ is the forward difference operator. A little more generally,
\[
\begin{aligned}
\Delta(\ff{x}{l}\ff{y}{m})& = (\Delta \ff{x}{l})\ff{y}{m} + \ff{x}{l}(\Delta\ff{y}{m}) + (\Delta \ff{x}{l})(\Delta\ff{y}{m})\\
&= l\ff{x}{{l-1}}\ff{y}{m} + m\ff{x}{l}\ff{y}{{m-1}}+lm\ff{x}{{l-1}}\ff{y}{{m-1}}\,.
\end{aligned}
\]
Then, on the one hand,
\[\begin{aligned}
\Delta \ff{(xy)}{k} &= \sum_{l,m=1}^k c^{(k)}_{l,m}( l\ff{x}{{l-1}}\ff{y}{m} + m\ff{x}{l}\ff{y}{{m-1}}+lm\ff{x}{{l-1}}\ff{y}{{m-1}}) \\ 
&=\sum_{l=0}^{k-1}\sum_{m=1}^k c^{(k)}_{l+1,m}(l+1)\ff{x}{l}\ff{y}{m}  + 
\sum_{l=1}^k \sum_{m=0}^{k-1}c^{(k)}_{l,m+1} (m+1)\ff{x}{l}\ff{y}{m} +  \sum_{l,m=0}^{k-1} c^{(k)}_{l+1,m+1}(l+1)(m+1)\ff{x}{l}\ff{y}{m}\,.
\end{aligned}
\]
On the other hand, apply the identity
\[
\ff{v}{k} = \frac{1}{v}(\ff{v}{{k+1}}+k\ff{v}{k})
\]
 with $v = (x+1)(y+1)$ to write
 \[
 \begin{aligned}
 \Delta \ff{(xy)}{k} &= \ff{((x+1)(y+1))}{k}-\ff{(xy)}{k}\\
 &= \frac{1}{(x+1)(y+1)}\left(\ff{((x+1)(y+1))}{{k+1}}+k\ff{((x+1)(y+1))}{{k}}\right) -\ff{(xy)}{k}\,,
 \end{aligned}
 \]
and then
\[
\begin{aligned}
 \Delta \ff{(xy)}{k} &= \sum_{l,m=1}^{k+1} c^{(k+1)}_{l,m}\frac{\ff{(x+1)}{l}}{x+1}\frac{\ff{(y+1)}{m}}{y+1} + k\sum_{l,m=1}^k c^{(k)}_{l,m}\frac{\ff{(x+1)}{l}}{x+1}\frac{\ff{(y+1)}{m}}{y+1} - \sum_{l,m=1}^k c^{(k)}_{l,m} \ff{x}{l}\ff{y}{m}\\
 &= \sum_{l,m=1}^{k+1}c^{(k+1)}_{l,m}\ff{x}{{l-1}}\ff{y}{{m-1}} + k \sum_{l,m=1}^{k}c^{(k)}_{l,m}\ff{x}{{l-1}}\ff{y}{{m-1}}-\sum_{l,m=1}^k c^{(k)}_{l,m} \ff{x}{l}\ff{y}{m}\\
 &= \sum_{l,m=0}^k c^{(k+1)}_{l+1,m+1}\ff{x}{l}\ff{y}{m} + k\sum_{l,m=0}^{k-1} c^{(k)}_{l+1,m+1}\ff{x}{l}\ff{y}{m} -  \sum_{l,m=1}^k c^{(k)}_{l,m} \ff{x}{l}\ff{y}{m}\,.
 \end{aligned}
 \]

Comparing the two expressions for $ \Delta \ff{(xy)}{k} $, and rearranging terms, we have shown:

\begin{theorem} \label{theorem:recurrence}
The coefficients $c^{(k)}_{l,m}$ satisfy the recurrence
\[
c^{(k+1)}_{l+1,m+1} = c^{(k)}_{l,m}+(l+1)c^{(k)}_{l+1,m}+(m+1)c^{(k)}_{l,m+1}+((l+1)(m+1)-k)c^{(k)}_{l+1,m+1}\,.
\]
\end{theorem}
 
We comment that other proofs of the recurrence are possible, for example one that uses the recurrence relations for Stirling numbers.

\section{Combinatorial Characterizations} \label{section:combinatorics}

Conjoint analysis is a method in marketing that assigns weightings to independent attributes of a product. It is only the first step that we consider, that of setting up a conjoint ranking table. Suppose there are $l\ge 1$ choices for one attribute (e.g. price) and $m\ge 1$ for the other (e.g. color). Form an $l \times m$ table, in which each cell represents the two attributes considered jointly; hence the contraction `conjoint.'  Let $ k \le lm$ and rank the conjoint preferences from $1$ to $k$. 
Since we allow $k \le lm$ not every pair of attributes need have a ranking, but we do insist that every individual attribute must enter into the ranking at least once. That is, every row or  column must have at least one ranked cell. In particular $k \ge \max\{l,m\}$. One can consider the remaining cells as left blank or filled in with zeros.
 
For a given $l$ and $m$ we refer to such an object as a  $k$-\emph{conjoint ranking table}. This is quite a general concept and it is easy to imagine many examples. Here is one other. A graduate student is signing up to take qualifying exams, and is required to fill out an $l \times m$ table indicating preferences. Each of the $l$ rows refers to  a subject area, such as algebra, combinatorics, analysis, etc.. Each of the $m$ columns refers to a particular professor; we assume that they are all equally capable of asking about any of subjects. The student will be asked $k$ questions, where $k \le lm$, and may put the numbers $1$ through $k$ in any of the cells, indicating
 preferences for \textit{who} asks \textit{what kinds} of questions. However, the student is  not allowed to avoid any professors, and must put at least one number in each column. Similarly, the student is not allowed to avoid any subject areas, and must put at least one  number in each row. 

The recurrence relation allows us to derive a combinatorial  interpretation and characterization of the $c$'s as counting the number of such tables.

\begin{theorem} \label{theorem:conjoint-count}
The number of  $k$-conjoint ranking tables of size $l\times m$ is $l!\,m!\,c^{(k)}_{l,m}$.
\end{theorem}

Before giving the proof we observe the following consequence.
\begin{corollary}
The $c^{(k)}_{lm}$ satisfy
\[
\label{eq:positivity-2}
c_{l,m}^{(k)} = 0 \quad  \text{for} \quad  lm <k \quad  \text{and} \quad c_{l,m}^{(k)} > 0\quad \text{for} \quad l m \ge k\,.
\]
\end{corollary} 

This is so because when $lm < k$, the conjoint ranking table is too small to fit all $k$ numbers and so $c^{(k)}_{l,m}=0$. However, when $lm \geq k$, there are enough cells in the table to be ranked from $1$ to $k$. Furthermore, there must exist at least one way of placing the numbers so as to satisfy the row and column constraints since the conditions $1 \leq l,m \leq k$ imply that $k \geq \max(l,m)$. Thus $c^{(k)}_{l,m} >0$.

One can see the pattern of zeros  for $k=9$ in \eqref{eq:C9}. The parabolic shape of the boundary between the zero and nonzero coefficients becomes more pronounced as $k$ increases.

\begin{proof}[Proof of Theorem \ref{theorem:conjoint-count}]
To prove the theorem we first define an equivalence relation on the set of $k$-conjoint ranking tables, namely two tables of the same size are equivalent if one is obtained from the other by permuting the rows, or columns, or both. Let  $d_{l,m}^{(k)}$ be the number of equivalence classes. We show that  $d_{l,m}^{(k)} = c_{l,m}^{(k)}$ by showing that the same recurrence relation obtains. 

The following four cases are properties of any table in an equivalence class:
\begin{enumerate}
\item $k$ is alone in both its row and column; 
\item $k$ is alone in its row but not its column;
\item $k$ is alone in its column but not its row;
\item $k$ is  neither alone in its row nor its column. 
\end{enumerate}
These cases are illustrated in Figure \ref{fig-combinatorial}. We count the number of equivalence classes by computing how much each case contribute to $d_{l,m}^{(k)}$.

\begin{figure} [ht]
\centering
\includegraphics{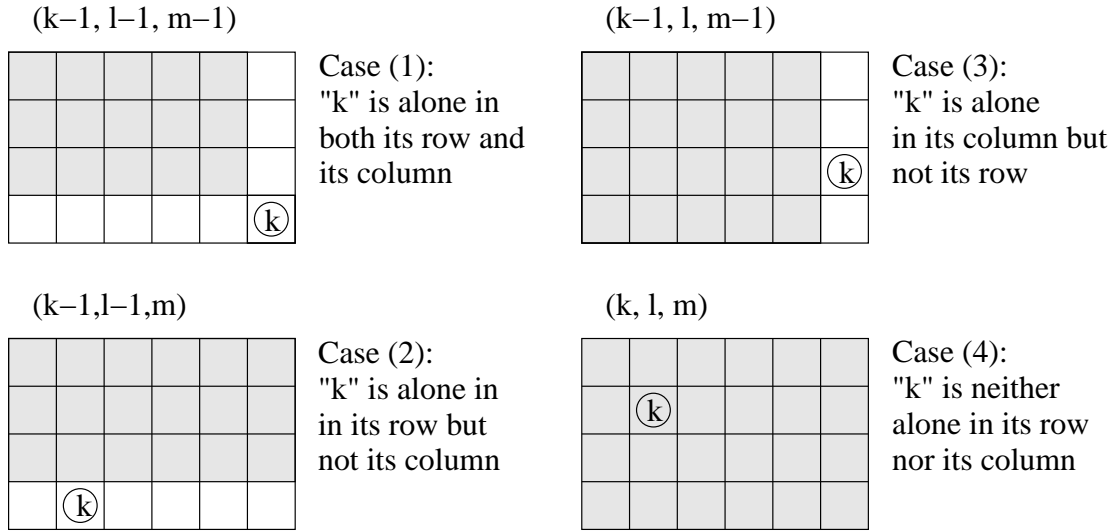}
\caption{Diagrams illustrating the four cases in the proof of Theorem \ref{theorem:conjoint-count}.}
\label{fig-combinatorial}
\end{figure}

\begin{itemize}

\item[] Case (1): If we  remove both the row and column containing $k$, then we are left with a $(k-1)$-conjoint ranking table of size $(l-1)\times (m-1)$. Hence this case contributes  $d_{l-1,m-1}^{(k-1)}$.

\item[] Case (2): If we remove the row containing $k$ then what remains is a $(k-1)$-conjoint ranking table of size $(l-1)\times m$. Since the number $k$ could have been in any of the $m$ slots of the row just removed, the contribution is $m d_{l-1,m}^{(k-1)}$. 

\item[] Case (3): The analysis is identical to that of Case (2) and contributes $ld^{(k-1)}_{l,m-1}$.

\item[] Case (4): Whereas in previous cases we removed slices of a $k$-conjoint ranking table, now we will replace a zeroed entry of a $(k-1)$-conjoint ranking table of size $l\times m$ with $k$. There are $lm - (k-1)$ zeroed entries in such a table, so this case contributes  $(lm - (k-1)) d_{l,m}^{(k-1)}$. 
\end{itemize}
Adding the contributions from each of the four cases yields 
$$d_{l,m}^{(k)} = d_{l-1,m-1}^{(k-1)} + m d_{l-1,m}^{(k-1)} + l d_{l,m-1}^{(k-1)} + (lm - (k-1)) d_{l,m}^{(k-1)}$$
which is the same recurrence as for $c_{l,m}^{(k)}$ (reindexed). Considering small tables verifies that $d_{l,m}^{(k)}$ starts out like $c_{l,m}^{(k)}$, and the two quantities are therefore equal.

Since we allowed for permuting $l$ rows and $m$ columns, the number of $k$-conjoint ranking tables of size $l \times m$ is $l!\,m!\,c^{(k)}_{l,m}$.

\end{proof}

Here we make contact with earlier work, for the core of the proof above is counting a set of binary matrices, and these have been counted in other ways; see Chapter 9 in \cite{charalambides:combinatorics}. One approach is to use the inclusion-exclusion principle. We would like to give this argument to show how it leads to another expression for $c^{(k)}_{l,m}$ (which also appears in older literature, but not as interpreted here).

Fix $l$ and $m$, take $k \le lm$, and let $\mathcal{O}$ be the set of   binary matrices of size $l \times m$ with $1$'s in exactly $k$ positions. Then
\[
|\mathcal{O}| = \binom{lm}{k}\,.
\]
Now let $k \ge \max\{l,m\}$ and let  $\mathcal{C}$ be the subset of $\mathcal{O}$ which have at least one $1$ in every row and in every column. We want to find $|\mathcal{C}|$.

For the index $i$ running from $1$ to $l$ let $\mathcal{A}_i$ be the subset of $\mathcal{O}$ whose $i$'th row has all $0$'s, and for the index $i$ running from $1$ to $m$ let $\mathcal{A}_{i+l}$ be the subset of $\mathcal{O}$ whose $i$'th column has all $0$'s. Then
\[
\mathcal{C} = \overline{\mathcal{A}_1}\cap\overline{\mathcal{A}_2}\cap\cdots \cap\overline{\mathcal{A}_{l+m}}\,,\]
where
\[
\overline{\mathcal{A}_i} = \mathcal{O}\setminus \mathcal{A}_i\,.
\]
By the inclusion-exclusion principle
\[
|\mathcal{C}| = |\mathcal{O}|-\sum_{i=1}^{l+m} |\mathcal{A}_i| + \sum_{i_1<i_2}|\mathcal{A}_{i_1}\cap\mathcal{A}_{i_2}|-\sum_{i_1<i_2<i_3}|\mathcal{A}_{i_1}\cap\mathcal{A}_{i_2}\cap\mathcal{A}_{i_3}|+\cdots + (-1)^{l+m}|\mathcal{A}_1\cap\mathcal{A}_2\cap\cdots\cap\mathcal{A}_{l+m}|\,.
\]
To compute the general sum
\[
\sum|\mathcal{A}_{i_1}\cap \cdots \cap \mathcal{A}_{i_h}|
\]
we have to have to distinguish the matrices that have zeroed rows from those that have zeroed columns. Suppose among the $h$ sets $\mathcal{A}_{i_1},\dots,\mathcal{A}_{i_h}$ that $p$ of them have zeroed rows. Then $h-p$ have zeroed columns. For a fixed $p$ there are $\binom{l}{p}$ ways to select $p$ rows to zero out and there are $\binom{m}{h-p} = \binom{m}{m-h+p}$ ways to select $h-p$ columns to zero out.  After these choices there are $pm+(h-p)l-p(h-p)$ zeros total, and so there remain $lm-pm-(h-p)l+p(h-p) = (l-p)(m-h+p)$ cells amongst which we place $k$ ones. There are then
\[
\binom{(l-p)(m-h+p)}{k}
\]
ways of doing this. 
Thus
\[
\sum|\mathcal{A}_{i_1}\cap \cdots \cap \mathcal{A}_{i_h}|=\sum_{p=0}^h \binom{l}{p}\binom{m}{m-h+p} \binom{(l-p)(m-h+p)}{k}
\]
and 
\[
|\mathcal{C}| = \sum_{h=0}^{l+m}\sum_{p=0}^h(-1)^h \binom{l}{p}\binom{m}{m-h+p} \binom{(l-p)(m-h+p)}{k}\,.
\]

 Multiplying $|\mathcal{C}|$ by $k!$ distinguishes the nonzero elements, whether they are $k$ distinguished balls tossed into bins or the numbers from $1$ to $k$.  Then dividing by $l!m!$ allows for permuting the rows and columns. The end result is evidently the same as counting the number of $k$-conjoint ranking tables, and hence
\begin{equation}  \label{eq:formula-2}
c^{(k)}_{l,m} = \frac{k!}{l!m!}|\mathcal{C}| = \frac{k!}{l!m!} \sum_{h=0}^{l+m}\sum_{p=0}^h(-1)^h \binom{l}{p}\binom{m}{m-h+p} \binom{(l-p)(m-h+p)}{k}\,.
\end{equation}

There is one more approach and one more formula. The equation \eqref{eq:main} can also be written in the form
\begin{equation} \label{eq:binomial-product}
 \binom{xy}{k}= \sum_{l,m=1}^{k}  b^{(k)}_{l,m} \binom{x}{l}\binom{y}{m}\,,
\end{equation}
where
\[
 b^{(k)}_{l,m} = \frac{l!m!}{k!}c^{(k)}_{l,m}\,.
 \]
We understand the binomial coefficient to be defined for nonintegral $x$ by
\[
\binom{x}{k} = \frac{\Gamma(x+1)}{k!\Gamma(x-k+1)}\,,
\]
and we use
\[
x(x-1)\ldots(x-k+1)=\frac{\Gamma(x+1)}{\Gamma(x-k+1)}\,.
\]

We see that the  $b^{(k)}_{l,m}$ give a count of the number of the $l \times m$ binary matrices with exactly $k$ ones such that each row and column has at least one $1$.  We have a further comment on the combinatorics of \eqref{eq:binomial-product} when $x$ and $y$ are integers. The left-hand side, $\binom{xy}{k}$ is just the number of ways to select $k$ elements from an $x \times y$ array. How does this jibe with the right-hand side?  

For any given selection of $k$ cells there is a unique minimal subarray (smallest number of rows and columns) such that every row and column of that subarray contains a selected cell. This is illustrated in Figure \ref{fig-binomialidentity}. On the left, the darkened cells correspond to a selection of $k=8$ cells, and the arrows indicate the rows and columns of the subarray in which the selected cells are contained. On the right we see the subarray extracted, and note that every row and column contains a selected square. Counting all possible such subarrays thus counts the ways of choosing $k$ elements from the big array. This is what the  right-hand side in equation \eqref{eq:binomial-product}  does, for $\binom{x}{l}\binom{y}{m}$ is the number of ways to choose an $l \times m$ subarray, and then  for each such subarray we count the number of ways to populate it with $k$ entries such that no row or column is void. That multiplier is precisely  $b_{l,m}^{(k)}$. 

\newpage

\begin{figure} [ht]
\centering
\includegraphics{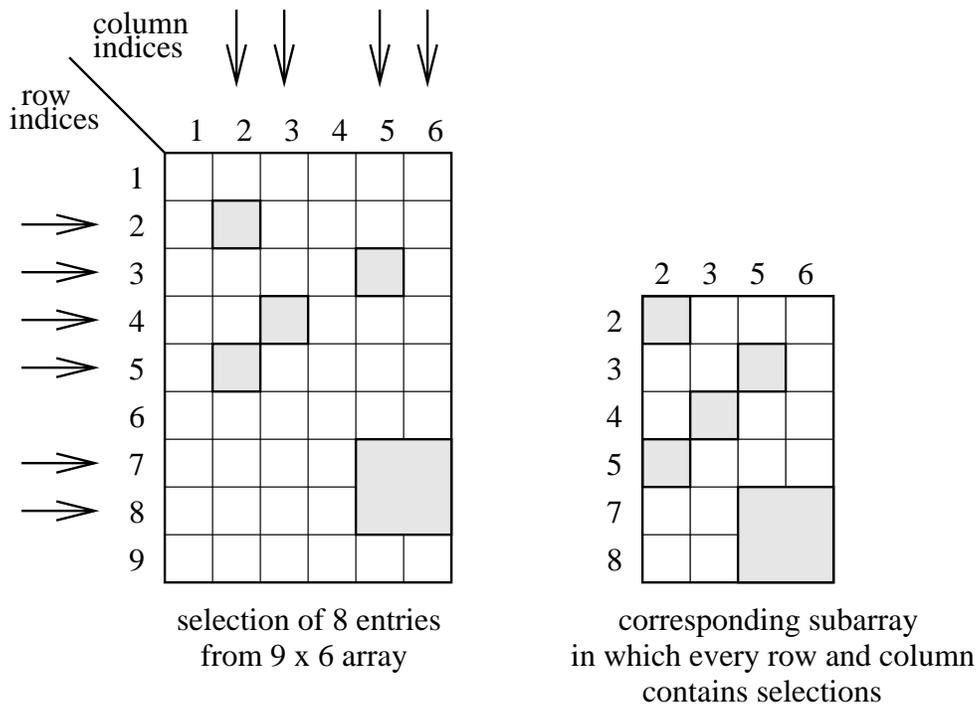}
\caption{Combinatorial Interpretation of  equation \eqref{eq:binomial-product}.}
\label{fig-binomialidentity}
\end{figure}

From Theorem \ref{theorem:recurrence}, the $b$'s satisfy the recurrence
\begin{equation} \label{eq:recurrence-b}
b^{(k+1)}_{l+1,m+1}= (l+1)(m+1)(b^{(k)}_{l,m}+b^{(k)}_{l,m+1}+b^{(k)}_{l+1,m})+((l+1)(m+1)-k)b^{(k)}_{l+1,m+1}\,.
\end{equation}
But explicitly in \cite{cameron:asymptotics} it was noted that M\"obius inversion applied to \eqref{eq:binomial-product} implies
\begin{equation} \label{eq:formula-3}
b^{(k)}_{l,m} = \sum_{s,t=1}^k (-1)^{l+s+m+t}\binom{l}{s}\binom{m}{t}\binom{st}{k}\,.
\end{equation}
Furthermore, a similar formula seems also to be found in \cite{maia-mendez:species}. M\"obius inversion is kin to matrix inversion and it is interesting to see how the latter can be used to derive \eqref{eq:formula-3}. 

For $k^2$ pairs $(x_i,y_i)$, $i = 1, \dots ,k^2$, we treat \eqref{eq:binomial-product} as a system of linear equations for the $k^2$ unknowns $b^{(k)}_{l,m}$. In matrix form, $A \underline{\beta}=\underline{\xi}$ where 
\[
\begin{aligned}
& \underbrace{\left[
\begin{array}{cccccccc}
\binom{x_1}{1}\binom{y_1}{1} & \binom{x_1}{1}\binom{y_1}{2} & \cdots & \binom{x_1}{1}\binom{y_1}{k} & \binom{x_1}{2}\binom{y_1}{1} & \binom{x_1}{2}\binom{y_1}{2} & \cdots & \binom{x_1}{k}\binom{y_1}{k}\\
\binom{x_2}{1}\binom{y_2}{1} & \binom{x_2}{1}\binom{y_2}{2} & \cdots & \binom{x_2}{1}\binom{y_2}{k} & \binom{x_2}{2}\binom{y_2}{1} &\binom{x_2}{2}\binom{y_2}{2} & \cdots & \binom{x_2}{k}\binom{y_2}{k} \\
\vdots & \vdots &\cdots & \vdots & \vdots &\vdots & \vdots & \vdots\\
\binom{x_{k^2}}{1}\binom{y_{k^2}}{1} & \binom{x_{k^2}}{1}\binom{y_{k^2}}{2} & \cdots &\binom{x_{k^2}}{1}\binom{y_{k^2}}{k} & \binom{x_{k^2}}{2}\binom{y_{k^2}}{1} &  \binom{x_{k^2}}{2}\binom{y_{k^2}}{2}  & \cdots & \binom{x_{k^2}}{k}\binom{y_{k^2}}{k}
\end{array}
\right]}_A
\underbrace{\left[ \begin{array}{c}
b^{(k)}_{1,1} \\
b^{(k)}_{1,2} \\
\vdots \\
b^{(k)}_{1,k}\\
b^{(k)}_{2,1}\\
b^{(k)}_{2,2}\\
\vdots\\
b^{(k)}_{L(i),M(i)}\\
\vdots\\
b^{(k)}_{k,k}
\end{array} \right]}_{\underline{\beta}}
\\
& \hspace{3in} =
\underbrace{\left[ \begin{array}{c}
{{x_1 y_1} \choose k} \\
{{x_2 y_2} \choose k} \\
\vdots \\
{{x_iy_i}}\choose{k}\\
\vdots\\
{{x_{k^2} y_{k^2}} \choose k} 
\end{array} \right]}_{\underline{\xi}}
\end{aligned}
\]
Here $L(i)$ and $M(i)$ are mappings that convert from linear indexing to matrix indexing:
\[
L(i) = \left\lfloor\frac{i-1}{k}\right\rfloor +1\,,\quad M(i) = ((i-1)\mod k)+1\,,\quad i = 1, \dots, k^2\,.
\]
To go the other way,
\[
i=k(l-1)+m\,,\quad l,m = 1,\dots, k\,.
\]
Thus, succinctly,
\[
A_{i,j} = \binom{x_i}{L(i)}\binom{y_i}{M(j)}\,,\quad \underline{\beta}_i=b^{(k)}_{L(i),M(i)}\,,\quad \underline{\xi}_i=\binom{x_iy_i}{k}\,.
\]

We want to choose the $x_i,y_i$ to make $A$ invertible, and solve for $A^{-1}$. We find that if we take $x_i=L(i)$ and $y_i=M(i)$, so that
\[
A_{i,j} = \binom{L(i)}{L(j)}\binom{M(i)}{M(j)}\,,
\]
then $A$ is invertible with 
\begin{equation} \label{eq:A-inverse}
\det A = 1 \,,\quad A^{-1}_{i,j} = (-1)^{L(i)+L(j)+M(i)+M(j)}A_{i,j}\,.
\end{equation}
This closely mimics the well-known phenomenon exhibited by the Pascal matrix and its inverse:
\[
\label{eq-pascalinversion}
P_{i,j} = {{i} \choose {j}}, \qquad P^{-1}_{i,j} = (-1)^{i+j}{{i} \choose {j}}.
\]
In \cite{ow:pascal} we prove a more general result so we will not give the details behind \eqref{eq:A-inverse} here. Briefly, the first step is to observe that $A$ is the element-wise product (Hadamard product) of the two matrices $A^{(x)}$ and $A^{(y)}$, where 
\[
A^{(x)}_{i,j} = \binom{L(i)}{L(j)}\,,\quad A^{(y)}_{i,j} = \binom{M(i)}{M(j)}\,,\]
and $A^{(x)}$ is, as a whole,  block lower-triangular with $k \times k$ size blocks while in $A^{(y)}$ the separate $k\times k$ blocks are each lower triangular. See Figure \ref{fig-triangularMatrices}. It follows that  $A$ is lower-triangular.

\begin{figure}[htbp]
\begin{center}
\includegraphics[scale=0.65]{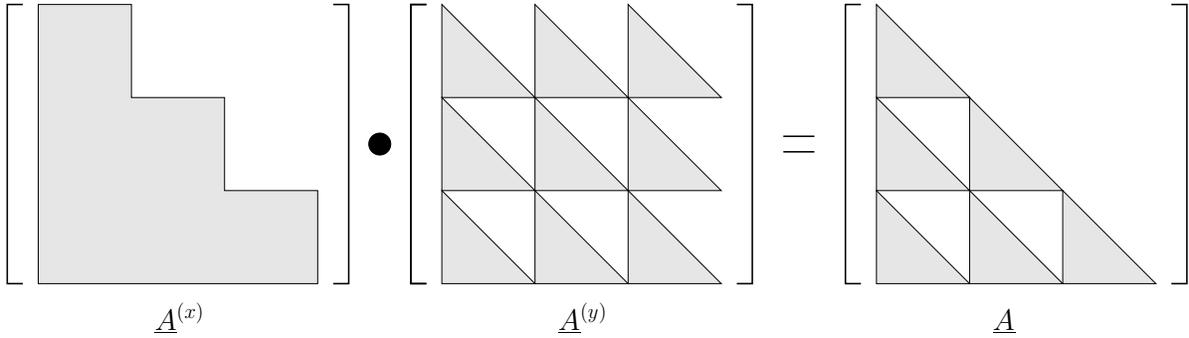}
\caption{Shapes of $A^{(x)}$, $A^{(y)}$ and $A$. The matrix $A$ on the right is the Hadamard product of the matrices on the left.}
\label{fig-triangularMatrices}
\end{center}
\end{figure}

The diagonal elements of $A$ are all $1$ so $\det A=1$. Defining $B=(-1)^{L(i)+L(j)+M(i)+M(j)}A_{i,j}$ one then shows directly that $BA$ is the identity matrix, and this requires some special identities for binomial coefficients. 

From the formula for $A^{-1}$ let us derive \eqref{eq:formula-3}. We have 
\[
\underline{\beta}= A^{-1}\underline{\xi}\,,\quad 
b^{(k)}_{l,m}  = \beta_i\,,\quad \text{where $i=k(l-1)+m$,}
\]
or
\[
\begin{aligned}
b^{(k)}_{l,m} &= \beta_i = \sum_{j=1}^{k^2} A_{i,j}^{-1}\xi_j = \sum_{j=1}^{k^2} A_{i,j}^{-1}\binom{L(j)M(j)}{k}\\
&=  \sum_{j=1}^{k^2} (-1)^{L(i)+L(j)+M(i)+M(j)}A_{i,j}\binom{L(j)M(j)}{k}\\
&=  \sum_{j=1}^{k^2} (-1)^{L(i)+L(j)+M(i)+M(j)}\binom{L(i)}{L(j)}\binom{M(i)}{M(j)}\binom{L(j)M(j)}{k}\,.
\end{aligned} 
\]

For $L(i)$ and $M(i)$, using  $1 \leq m \leq k$ it is easy to verify that
\begin{align*}
L(i) &= \floor{\frac{i-1}{k}} + 1 = \floor{\frac{k(l-1) + m -1}{k}} + 1 =  l \\
M(i) &= ((i-1) \mod k) + 1 = ((k(l-1) + m-1) \mod k) + 1 = m.
\end{align*}
To simplify the $L(j)$ and $M(j)$ terms, express $j$ as the combination
\[
j = ks + t, \qquad s \in \{0,1,\ldots,k-1\}, \quad t \in \{ 1,2,\ldots,k\}\,.
\]
Then just as above,
\begin{align*}
L(j)  &= L(ks + t) = \floor{\frac{ks+t-1}{k}} +1  = \floor{s + \frac{t-1}{k}} + 1 = s+1 \\
M(j) &=M(ks+t) = ((ks+t-1) \mod k) + 1 = (t-1) \mod k + 1 = t-1 +1 = t.
\end{align*}
Hence,
\begin{align*}
b_{l,m} &= \sum_{j=1}^{k^2} (-1)^{L(i) + L(j) + M(i) + M(j)}  {{L(i)} \choose {L(j)}} {{M(i)} \choose {M(j)}}  {{L(j) M(j)} \choose k} \\
&= \sum_{s=0}^{k-1} \sum_{t=1}^k (-1)^{L(i) + L(ks+t) + M(i) + M(ks+t)}  {{L(i)} \choose {L(ks+t)}} {{M(i)} \choose {M(ks+t)}}  {{L(j) M(ks+t)} \choose k} \\
&= \sum_{s=0}^{k-1} \sum_{t=1}^k (-1)^{l + s+1 + m + t}  {{l} \choose {s+1}} {{m} \choose {t}}  {{(s+1)t} \choose k} \\
&= \sum_{s=1}^{k} \sum_{t=1}^k (-1)^{l + s + m + t}  {{l} \choose {s}} {{m} \choose {t}}  {{st} \choose k}\,,
\end{align*}
which is \eqref{eq:formula-3}.

\bigskip

We conclude this part of the discussion by noting that we have three different expressions for the $c$'s (or the $b$'s), \eqref{eq:c-formula-1}, \eqref{eq:formula-2} and \eqref{eq:formula-3} -- and we do not know algebraically  how to derive one from another!

\section{More than two Variables}

For an expansion
\[
\ff{(x_1x_2\cdots x_n)}{k} = \sum_{L} c^{(k)}_L\ff{{x_1}}{{l_1}}\ff{{x_2}}{{l_2}}\cdots\ff{{x_n}}{{l_n}}\,, \quad L = (l_1, l_2, \dots, l_n)\,, 1\le l_i \le k\,,
\]
of a product of more than two variables all the results in the preceding sections have natural extensions. Modifications to the earlier arguments are straightforward and so we  record the outcomes with little additional detail -- the chief problem is notation. We follow the generally accepted conventions on multi-indexing. In particular
\[
L! = l_1!l_2!\cdots l_n! \,.
\]

The formula for the $c$'s in terms of Stirling numbers is
\[
c^{(k)}_L = \sum_{p=1}^k (-1)^{k-p}\stone{k}{p}\prod_{i=1}^n \sttwo{p}{l_i}\,.
\]
or simply
\[
c^{(k)}_L = \sum_{p=1}^k (-1)^{k-p}\stone{k}{p}\sttwo{p}{L}\,.
\]
if we allow ourselves the analog to the the multi-indexed case of binomial coefficients and write
\[
\sttwo{p}{L}=\prod_{i=1}^n \sttwo{p}{l_i}\,.
\]

There is a natural extension of the product rule for the forward difference operator and it can be applied just as before to obtain a recurrence relation. It will pay to invest in a little extra notation. We write 
\[
L+1=(l_i+1\colon i=1,\dots,n)
\]
and for a subset $S \subseteq \{1,\dots ,n\}$ we write
\[
L_S = (l_i\colon i \in S)\,,\quad L_{S}+1 = (l_i+1 \colon i \in S)\,,\quad L_{\overline{S}}=(l_i\colon i \in \{1,\dots ,n\}\setminus S)\,.
\]
The result is
\[
c^{(k+1)}_{L+1}=c^{(k)}_{L} -kc^{(k)}_{L+1}+
\sum_{m=1}^n\sum_{|S|=m}(\prod_{i\in S} l_i)c^{(k)}_{L_S+1,L_{\overline{S}}}\,.
\]

The generalization of a 2-dimensional conjoint ranking table allows for $n$ independent attributes (color, price, shape, \dots) with $l_i$ choices for the $i$'th attribute. Based on the recurrence one can then show that $L!c^{(k)}_L$ is the number of ways to fill in an $l_1\times l_2 \times \cdots \times l_n$ conjoint ranking table with the numbers $1$ through $k$, insisting, as before,  that each attribute must enter into the ranking at least once. Again this implies the nonnegativity of the $c$'s, more precisely
\[
\text{$c^{(k)}_L=0$ if $\prod_i l_i <k$ and $c^{(k)}_L>0$ if $\prod_i l_i  \ge k$.}
\]

Extensions of the alternate formulas \eqref{eq:formula-2} and \eqref{eq:formula-3} are more complicated to write. For the former, to keep the final result from being too cluttered we use the notation
\[
\|L\| = \sum_{l_i\in L} l_i 
\]
and for a multi-index $M = (m_1,m_2, \dots, m_n)$ with $1 \le m_i \le h$ and  $\|M\| = h$ we let
\[
\Phi(L,M) = \prod_{i=1}^n l_i - \sum_{r=1}^n \left(l_r\prod_{s < r}(l_s-m_s)\prod_{s>r}l_s\right)
\]
Then
\[
c^{(k)}_L=\frac{k!}{ L!} \sum_{h=0}^{\|L\|} (-1)^h\sum_{M,\|M\|=h} \binom{\Phi(L,M)}{k}\prod_{j=1}^n\binom{l_j}{m_j}\,.
\]

 While the derivation of this formula is only an extension of the argument for two variables, some discussion will help make the form clearer. Let ``slices" refer to the higher-dimensional generalization of the notion of rows (or columns) for matrices. As before, we use inclusion-exclusion to count $b^{(k)}_L$, the number of distinct $l_1 \times \ldots \times l_n$ $(0,1)$-tables with exactly $k$ ones and no zeroed slices; that's the formula without the factorials in front. In the outer summation, the index $h$ is the number of zeroed slices. The inner summation then runs over possible ways to distribute the $h$ zeroed slices across the $n$ different dimensions. $m_i$  is the number of zeroed slices in dimension $i$ and $\binom{l_i}{m_i}$ is the number of ways to select a particular set of $m_i$ slices to zero out. After counting ways to zero the slices, we must then count the number of ways to fill the remaining cells in the table with $k$ ones. The total number of cells in the array is $\prod_{i=1}^n l_i$ and $\Phi(L,M)$ counts the remaining  cells by subtracting the number of zeroed slices from the total number of cells. Then $\binom{\Phi(L,M)}{k}$ is the number of ways to distribute $k$ ones among these remaining cells. 
 
 The analog of the formula \eqref{eq:formula-3} based on matrix inversion is developed in \cite{ow:pascal}. It reads
 \[
 c^{(k)}_L=\frac{k!}{L!}\sum_{r_1,r_2,\dots,r_n=1}^k (-1)^{\sum_i(r_i+l_i)}\binom{\prod_ir_i}{k}\prod_i\binom{l_i}{r_i}\,.
 \]

\bibliography{fallingfactorials}
\end{document}